\documentclass[11pt]{amsart}

\usepackage{amsmath, amsthm, amssymb}
\usepackage{bbm}
\usepackage[hidelinks]{hyperref}

\numberwithin{equation}{section}

\newtheorem{theorem}{Theorem}[section]

\newtheorem{lemma}[theorem]{Lemma}

\newtheorem{question}[theorem]{Question}

\theoremstyle{definition}
\newtheorem{example}[theorem]{Example}
\newtheorem{remark}[theorem]{Remark}

\title[Erratum to  ``Computation of maximal  projection constants'']{Erratum to ``Computation of maximal \\ projection constants"}
\author{Giuliano Basso}
\address{Max Planck Institute for Mathematics, Vivatsgasse 7, 53111 Bonn, Germany}
\email{basso@mpim-bonn.mpg.de}

\begin{document}

\begin{abstract}
This is an erratum to the article: “Computation of maximal projection constants" (J. Funct. Anal., 277).  The statement of Lemma~3.1(2) of that paper is incorrect. As a consequence of this the proof of Theorem~1.4 is incomplete. In this erratum we prove a corrected version of Lemma~3.1 and explain why, with this weaker result, our original strategy for proving Theorem~1.4 no longer works. The other results of the article, in particular the alternative proof of Grünbaum's conjecture, do not depend on Lemma~3.1 and are thus not affected by this error.   
\end{abstract}

\maketitle

\setcounter{section}{1}

The second item of Lemma~3.1 in \cite{zbMATH07104037} is false in general. This error does not affect the main body of the article. However, Lemma~3.1 was the main tool we used in our proofs of Theorems~1.4 and 3.2. As it turns out, the proof strategy used in \cite{zbMATH07104037} to prove Theorem~1.4 cannot possibly lead to the claimed result. We will explain this in detail below. Theorem~1.4 is therefore still an open question, which can be formulated as follows.

\begin{question}\label{qu:new}
Let \(n\) be a positive integer. Is it true that there exists an integer \(d> n\) and an \(n\)-dimensional  linear subspace \(E\subset \ell_\infty^d\) such that \(E\) has maximal projection constant amongst all \(n\)-dimensional Banach spaces? If yes, what is the value of the smallest possible such \(d\)?
\end{question}

We now proceed by stating a corrected version of Lemma~3.1 and explain why the proof strategy used in \cite{zbMATH07104037} is not sufficient to answer Question~\ref{qu:new}. In the following, we use the notation from \cite[Section 3]{zbMATH07104037}. The corrected version of Lemma~3.1 reads as follows.

\begin{lemma}\label{lem:corrected}
The following holds true in every Banach space \(E\).
\vspace{0.25em}
\begin{enumerate}
\item If there exist closed linear subspaces \(V\), \(U\subset E\) such that  \(V\) is finite-dimensional and \(E=V\oplus U\), then 
\[
E^\ast=V^0\oplus U^0,\] 
and \(\dim(U^0)=\dim(V)\), as well as \((V^0)_0=V \textrm{ and } (U^0)_0=U\). \vspace{1em}

\item If there exist weak-star closed linear subspaces \(F\), \(G\subset E^\ast\) such that \(F\) is finite-dimensional and \(E^\ast=F\oplus G\), then 
\[
E=F_0\oplus G_0,
\]
and \(\dim(G_0)=\dim(F)\), as well as
\((F_0)^0=F \textrm{ and } (G_0)^0=G\).  \vspace{1em}
\item  If there exist closed linear subspaces \(V\), \(U\subset E\) such that \(V\) is finite-dimensional and \(E=V\oplus U\),  then
\begin{equation*}
\lVert P_V^U \rVert=\lVert P_{U^0}^{V^0} \rVert.
\end{equation*}
\end{enumerate}
\end{lemma}

\begin{proof}
\phantom{We prove each item separately:}
\begin{enumerate}
\item Since \(V\) and \(U\) are closed, it follows that \((V^0)_0=V \textrm{ and } (U^0)_0=U\), see e.g. \cite[Theorem~4.7]{zbMATH01022519}. Moreover, \(V\) and \(E/U\) are isomorphic as vector spaces, and thus using that \(V\) is finite-dimensional and \((E/U)^\ast=U^0\), we find that \(\dim(U^0)=\dim(V)\). To finish the proof we show in the following that \(E^\ast=V^0\oplus U^0\).
We abbreviate \(F=V^0\) and \(G=U^0\). Let \(\ell\in E^\ast\). Then \(\ell_F:=\ell\circ P_U^V\) and \(\ell_G:=\ell\circ P_V^U\) belong to \(F\) and \(G\), respectively. Moreover, if \(x=v+u\) with \(v\in V\) and \(u\in U\), then
\[
\ell_F(x)+\ell_G(x)=\ell(u)+\ell(v)=\ell(u+v)=\ell(x),
\]
and therefore \(F+G=E^\ast\). If \(\ell\in F\cap G\), then \(\ell(v)=0\) and \(\ell(u)=0\) for all \(v\in V\) and \(u\in U\). Since \(E=V\oplus U\), this implies that \(\ell=0\). Hence, \(E^\ast=F\oplus G\), as desired.\vspace{1em}

\item It is readily verified that \(F_0\cap G_0=\{0\}\). In the following we show that \(F_0\oplus G_0=E\). Since \(F\) and \(G\) are weak-star closed, it follows directly from the Hahn-Banach theorem that 
\[
(F_0)^0=F \text{ and } (G_0)^0=G,
\]
see e.g. \cite[Theorem~4.7]{zbMATH01022519}. Using that \((E/F_0)^\ast=(F_0)^0=F\) we find that \(E/F_0\) is finite-dimensional. Let \(\pi\colon E \to E/ F_0\) denote the quotient map. Since \(G_0+F_0=\pi^{-1}(\pi(G_0+F_0))\), it follows that \(G_0+F_0\) is a closed linear subspace of \(E\). 

Now, suppose \(\ell\in E^\ast\) vanishes on \(G_0+F_0\). Let \(\ell=\ell_F+\ell_G\) denote the decomposition induced by \(E^\ast=F\oplus G\). For every \(x\in F_0\) we have \(\ell(x)=0\) and \(\ell_F(x)=0\), and consequently, \(\ell_G(x)=0\). This implies that \(\ell_G\in (F_0)^0=F\). But, by definition \(\ell_G\in G\) and therefore \(\ell_G=0\). Analogously, we can deduce that \(\ell_F=0\) and thereby \(\ell=0\). Hence, as \(F_0+G_0\) is a closed linear subspace of \(E\), it follows that \(F_0+G_0=E\), as desired. \vspace{1em}

\item  Given \(x\in E\) let \(x=x_V+x_U\) be the unique decomposition induced by \(V \oplus U\). We abbreviate \(F=V^0\) and \(G=U^0\) and for  \(\ell\in E^\ast\) we write \(\ell=\ell_F+\ell_G\) for the decomposition induced by \(E^\ast=F\oplus G\).  Moreover, we write \(S(E)\) for the unit sphere of a Banach space. We have
\[
\lVert P_V^U\rVert=\sup_{x\in S(E)} \lVert x_V \rVert= \sup_{x\in S(E)} \sup_{\ell\in S(E^\ast)} \ell_G(x_V)
\]
and also
\[
\lVert P_G^F\rVert=\sup_{\ell\in S(E^\ast)} \lVert \ell_G \rVert= \sup_{\ell\in S(E^\ast)} \sup_{x\in S(E)} \ell_G(x_V).
\]
Hence the desired equality follows. 
\end{enumerate}
\end{proof}


Originally, the second item of the lemma only required that \(G\) is closed and not necessarily weak-star closed. However, with this weaker assumption the item becomes a false statement. The following counterexample is due to T. Kobos, to whom I am grateful for making me aware of this issue.

\begin{example}\label{ex:kobos}
Let \(F\subset \ell_\infty\) be the two-dimensional linear subspace spanned by the vectors
\[
f_1=(1, 0, 1, 1, \ldots) \quad \text{ and } \quad f_2=(0,1,1,1, \ldots)
\]
and let \(G\subset \ell_\infty\) by the weak-star closed subset \(G=\ker(\pi_1)\cap\ker(\pi_2) \), where \(\pi_i\colon \ell_\infty \to \mathbb{R}\) denotes the projection onto the \(i\)th coordinate. Clearly, \(\ell_\infty=F\oplus G\). We know from the second item of Lemma~\ref{lem:corrected} that \(\ell_1=F_0\oplus G_0\). We abbreviate \(V:=(F_0)_0\) and \(U:=(G_0)_0\). In the following we show that \(V\oplus U\neq c_0\) and thus the statement of Lemma~3.1(2) in \cite{zbMATH07104037} is incorrect. 

Notice that for \(i\geq 3\) any of the vectors \(a(i)=(1,1,0, \ldots, 0, -1, 0, 0, \ldots)\), where the \(i\)-th coordinate is equal to \(-1\), is contained in \(F_0\). Hence, if \(x=(x_1, x_2, x_3, \ldots)\) is contained in \(V\), then necessarily \(x_1+x_2=x_i\) for all \(i\geq 3\). But \(x\in c_0\) and thus \(V\) is equal the one-dimensional linear subspace of \(c_0\) spanned by the vector \((1, -1, 0, 0, \ldots)\). Clearly, \(G_0\subset \ell_1\) is equal to the two-dimensional subspace of \(\ell_1\) spanned by the vectors \((1, 0, 0, 0, \ldots)\) and \((0, 1, 0, 0, \ldots)\). This implies that \(U=\ker(\pi_1|_{c_0})\cap \ker(\pi_2|_{c_0})\). Hence, for example, the vector \((1, 0,0, \ldots)\) is not contained in \(V\oplus U\) and thus \(V\oplus U\neq c_0\).
\end{example}

In \cite{zbMATH07104037}, Lemma~3.1 was used in an essential way to prove Theorem~3.2. This theorem claimed that if \(X\) is a Banach space and \(F\subset X^{\ast\ast}\) a finite-dimensional subspace, then the pre-pre-annihilator \(V=(F_0)_0\) has the property that 
\begin{equation}\label{eq:claimed}
\lambda(V, X)=\lambda(F, X^{\ast\ast}),
\end{equation}
where \(\lambda\) denotes the relative projection constant. However, using the example from above, it is not difficult to show that this claimed equality cannot be true in general. Let \(X=c_0\) and consider the linear subspace \(F\subset X^{\ast\ast}=\ell_\infty\) from Example~\ref{ex:kobos}. We know that \(V:=(F_0)_0\) is \(1\)-dimensional and thus \(\lambda(V, X)=1\). But \(F\) is isometric to \(\mathbb{R}^2\) equipped with the norm 
\[
\lVert (x, y) \rVert=\max\{\vert x\vert, \vert y \vert, \vert x+y \vert\}
\]
and thus in particular it is not isometric to \(\ell_\infty^2\). Hence, \(\lambda(F, X^{\ast\ast})>1\) and the claimed equality \eqref{eq:claimed} does not hold. 

The error is caused precisely due to the fact that for a given \(F\subset X^{\ast\ast}\), the pre-annihilator \(F_0\subset X^\ast\) is in general not necessarily weak-star closed. A modified version of Theorem~3.2 reads as follows.

\begin{theorem}\label{thm:modified}
Let \(X\) be a Banach space and \(F\subset X^{\ast\ast}\) a finite-dimensional subspace such that \(F_0\subset X^\ast\) is weak-star closed. Then \(V:=(F_0)_0\subset X\) satisfies \(\dim(V)=\dim(F)\) and \(\lambda(V, X)=\lambda(F, X^{\ast\ast})\).
\end{theorem}

\begin{proof}
Since \(F_0\subset X^\ast\) has codimension \(n\) and is weak-star closed and thus in particular closed in the norm topology, there exists an \(n\)-dimensional linear subspace \(A\subset X^\ast\) such that \(X^\ast=F_0\oplus A\). Hence, using the second item of Lemma~\ref{lem:corrected}, we find that \(X=V\oplus A_0\) and \(\dim(V)=\dim(A)=n\), as well as \(V^0=F_0\). Since \((F_0)^0=F\), it follows that \((V^0)^0=F\). Suppose now \(U\subset X\) is a closed linear subspace such that \(X=V\oplus U\). Then, using the first and third item of Lemma~\ref{lem:corrected}, we obtain \(\lVert P_V^U \rVert=\lVert P_{U^{0}}^{V^0}\rVert=\lVert P_F^{G}\rVert\) for \(G:=(U^0)^0\). This implies \(\lambda(F, X^{\ast \ast})\leq \lambda(V, X)\). 

To finish the proof, we show the other inequality. Let \(J\colon X \to X^{\ast\ast}\) denote the canonical embedding of \(X\) into \(X^{\ast\ast}\). Then for every \(x\in V\), we have \(J(x)(\ell)=\ell(x)=0\) for all \(\ell\in V^0=F_0\). This implies that \(J(V)\subset (F_0)^0=F\). Since \(\dim(V)=\dim(F)\), we find that \(J(V)=F\). Hence, for every \(P\colon X^{\ast\ast}\to F\) the map \(Q:=J^{-1}\circ P \circ J\) is a projection of \(X\) onto \(V\) and \(\lVert Q \rVert \leq \lVert P \rVert\). This implies that \(\lambda(V, X) \leq \lambda(F, X^{\ast \ast})\), completing the proof.
\end{proof}

Thus, if \(F\subset X^{\ast \ast}\) is as in Theorem~\ref{thm:modified}, then \(F\) is contained in the image of \(X\) under the canonical embedding \(X \to X^{\ast\ast}\). In particular, in the special case when \(X=c_0\), Theorem~\ref{thm:modified} can only be applied to spaces that are already contained in \(c_0\). Therefore, the strategy used in \cite{zbMATH07104037} to prove Theorem~1.4 cannot possibly work, since we would have to assume the very thing that we want to prove. 


In the following remark, we explain how the error in our proof of Theorem~1.4 affects the results of the follow-up article \cite{basso--2021}.

\begin{remark}
    
To the author's knowledge, Theorem~1.4 has so far only been applied in the article \cite{basso--2021}. In particular, the moreover part of \cite[Theorem~1.4]{basso--2021} is equivalent to an affirmative answer of Question~\ref{qu:new}, so it should be considered an open question. We also used Theorem~1.4 (in a non-essential way) as a shortcut in the proof of \cite[Proposition~5.1]{basso--2021}. In the following, we give a short sketch of how the use of Theorem~1.4 can be avoided in the proof of \cite[Proposition~5.1]{basso--2021}. Thus, with the exception of the moreover part of \cite[Theorem~1.4]{basso--2021}, the results of \cite{basso--2021} hold true regardless of whether polyhedral maximizers exist or not, and are therefore not affected by our incorrect proof of Theorem 1.4.

For the rest of this remark we use the notation of \cite{basso--2021}. Let \(P_0\in \mathcal{P}_{n, m}\) be such that \(\lvert P_0 \rvert\) is a positive matrix and 
\[
\rho(\lvert P_0 \rvert)=\max_{P\in \mathcal{P}_{n, m}} \rho(\lvert P \rvert)=\Pi(n, m).
\]
The existence of such a matrix can  readily be verified by combining \cite[Theorem~1.4]{basso--2021} and \cite[Lemma~3.2]{basso--2021}. Let \(v\in \mathbb{R}^m_{>0}\), \(d\in \mathbb{N}\), and \(q_i=p_i/d\) for \(i=1, \ldots, m\), and \(\Lambda=\text{diag}(q_1, \ldots, q_m)\) be defined as in the proof of \cite[Proposition~5.1]{basso--2021}. Further, let \(S_0\in \mathcal{S}_m\) be such that \(\pi_n(\sqrt{\Lambda} S_0 \sqrt{\Lambda})\) is maximal amongst \(S\in \mathcal{S}_m\) and let \(S\in\mathcal{S}_d\) be the \((p_1, \ldots, p_m)\)-blow up of \(S_0\). Clearly, there exists a \(d\times n\) matrix \(U\) with columns \(u_i\) such that \(U^t U=\mathbbm{1}_{n\times n}\) and \(S u_i=\lambda_i u_i\). We set \(P=UU^t\). 

By construction, \(\pi_n(S)=\text{Tr}(SP)\), and \(S\) and \(P\) commute. We show in the following that moreover \(S=\text{Sgn}(P)\). Let \(r_i\) denote the \(i\)-th row of \(U\). Since \(S\) is a \((p_1, \ldots, p_m)\)-blow up of \(S_0\), there exists a partition \(A_1, \ldots, A_m\) of \(\{1, \ldots, d\}\) such that \(\# A_i=p_i\) and 
\(r_s=r_t=:w_i\) for all \(s\), \(t\in A_i\). For \(i=1, \ldots, m\) we define \(z_i=\sqrt{p_i} w_i\) and let \(V\) denote the \(m\times n\) matrix with rows \(z_i\). Notice that \(V^tV=\mathbbm{1}_{n \times n}\). Moreover, using \cite[Lemma~2.2]{zbMATH07104037}, we find that \(Q:=VV^t\) satisfies \(\text{Tr}(\sqrt{\Lambda} S_0 \sqrt{\Lambda} Q)=\pi_n(\sqrt{\Lambda} S_0 \sqrt{\Lambda})\) and so by \cite[Lemma~3.2]{basso--2021}, we have \(\text{Sgn}(Q)=S_0\). But \(\text{Sgn}(P)\) is clearly the \((p_1, \ldots, p_m)\)-blow up of \(\text{Sgn}(Q)\). Hence, \(\text{Sgn}(P)=S\), as claimed. In particular, \(\lvert P \rvert\) is a positive matrix. 

To finish the proof, we need to show that \(d \cdot \rho(\lvert P \rvert)\leq j^t \lvert P \rvert j+\epsilon\).
Since \(\rho(\lvert Q\rvert)=\rho(\lvert P \rvert)\) (which follows from a short computation) and 
\[
\pi_n(S)=\text{Tr}(SP)=\sum_{i, j}^d s_{ij} p_{ij} = \sum_{i, j}^d \lvert p_{ij} \rvert=j^t \lvert P \rvert j,
\]
we obtain that
\[
d\cdot \rho(\lvert P \rvert)-j^t \lvert P \rvert j = d\cdot \rho(\lvert Q \rvert)-\pi_n(S).
\]
Thus, using that \(\rho(\lvert Q \rvert)\leq \rho(\lvert P_0 \rvert)\) as well as \(\pi_n(S)=d\cdot \pi_n(\text{Sgn}(P_0) \Lambda)\) and Equation (13) of \cite{basso--2021}, we may conclude that
\[
d\cdot \rho(\lvert P \rvert)-j^t \lvert P \rvert j \leq d\cdot \big( \rho(\lvert P_0 \rvert)-\sqrt{q}^t \lvert P_0 \rvert \sqrt{q}\big).
\]
Now, exactly the same argument as in the proof of \cite[Proposition~5.1]{basso--2021} shows that \(d\cdot \rho(\lvert P \rvert) \leq j^t \lvert P \rvert j+\epsilon\), as desired. 

\end{remark}

We take the opportunity to record some additional comments that correct some minor oversights in \cite{zbMATH07104037}.

\begin{enumerate}
\vspace{1em}
    \item In the proof of Theorem~1.2 we  rely crucially on the equality case in von Neumann's trace inequality, but we have omitted to give a reference for this fact. A good reference is Theobald's article \cite{theobald}, which deals with real symmetric matrices and proves exactly the statement we need. See also the survey article \cite{carlsson} which discusses the history of the 'equality case' in von Neumann's trace inequality in detail. \vspace{1em}
    \item In the proof of Lemma~4.2 it seems worth to point out that \(\Lambda_Q\) is an diagonal matrix with positive entries. Hence, \(D:=\sqrt{\Lambda_Q}\sqrt{\Lambda_{Q^t}}\) is a diagonal matrix and \(\frac{1}{\text{Tr}(D)} D\) is contained in \(\mathcal{D}_d\) and therefore the inequality 
    \[
    1 \leq \text{Tr}(\sqrt{\Lambda_Q}\sqrt{\Lambda_{Q^t}})
    \]
    at the bottom of p. 3575 in \cite{zbMATH07104037} is a direct consequence of the equalities \(\pi_n(A\sqrt{\Lambda_Q}\sqrt{\Lambda_{Q^t}})=\pi_n(A \Lambda)\) and \(\pi_n(A\Lambda)=\max_{D\in\mathcal{D}_d} \pi_n(AD)\). \vspace{1em}

    \item  For each \(\epsilon \in [0,1]\) we define
    \[
    \mathcal{D}_d^{(\epsilon)}:=\{ D\in \mathcal{D}_d : \min d_{ii} \geq \epsilon \}.
    \]
Although Lemma~4.2 is only stated if \(\Lambda\in \mathcal{D}_d^{(0)}\) is non-singular, using the previous remark, it is not difficult to see that the conclusions of the lemma also hold if \(\Lambda\in \mathcal{D}_d^{(\epsilon)}\) and
\[
\pi_n(A \Lambda)=\max_{D\in\mathcal{D}_d^{(\epsilon)}} \pi_n(AD).
\]
\item In Section~4.3 it is claimed that  Lemmas 4.1 and 4.2 directly imply that 
    \[
    \max_{D\in \mathcal{D}_{N}} \pi_2( R_N D)=\pi_2(\tfrac{1}{N} R_N).
    \]
    Formally, however, Lemma~4.2 is only applicable if the maximizer \(D\in \mathcal{D}_N\) is invertible. To handle the case where the maximizer is a singular matrix, it suffices to note that
    \[
    \max_{D\in \mathcal{D}_{N}} \pi_2( R_N D)=\lim_{\epsilon \downarrow 0} \max_{D\in \mathcal{D}_{N}^{(\epsilon)}} \pi_2( R_N D)
    \]
    and to use the modified version of Lemma~4.2 from the previous remark. 
    \vspace{1em}

    \item In the proof of Lemma~4.4 a fixed point result from metric geometry was used to construct a matrix \(\Lambda\) with certain properties. However, the use of this fixed point theorem can easily be avoided. Indeed, for any \(D\in X\), the matrix
    \[
    \Lambda=\frac{1}{\lvert \text{Stab}(A) \rvert} \sum_{Q\in \text{Stab}(A) } QDQ^t
    \]
    clearly has the desired properties. 
\end{enumerate}
\bibliographystyle{alpha}
\bibliography{sample}

\end{document}